\RequirePackage{amsthm}
\documentclass{sn-jnl}
\usepackage{amsthm}
\usepackage{graphicx}%
\usepackage{multirow}%
\usepackage{amsmath,amssymb,amsfonts}%
\usepackage{mathrsfs}%
\usepackage[title]{appendix}%
\usepackage{xcolor}%
\usepackage{textcomp}%
\usepackage{manyfoot}%
\usepackage{booktabs}%
\usepackage{algorithm}%
\usepackage{algorithmicx}%
\usepackage{algpseudocode}%
\usepackage{listings}%
\usepackage{enumerate}%
\usepackage{enumitem}
\usepackage{xspace}
\usepackage{tikz}
\usetikzlibrary{fit}

\usepackage{subcaption}

\usepackage{mathrsfs}
\usepackage{amsfonts}
\usepackage{mathtools}

\usepackage[all]{xy}
\usepackage{hyperref}
\usepackage{framed}
\usepackage{tikz}
 \usepackage[
    backend=bibtex,
    maxbibnames=9,
    sorting=none
]{biblatex}
\newtheorem{thm}{Theorem}
\newtheorem{defn}{Definition}
\newtheorem{lem}{Lemma}
\newtheorem{cor}{Corollary}
\newtheorem{prop}{Proposition}

\newcommand{\bb}{\textcolor{gray}{?}}
\newcommand{\lllp}{\vline}

\newcommand{\xodd}{X_{odd}}
\newcommand{\xconst}{X_{s}}
\newcommand{\bx}{\mathbf{x}}
\newcommand{\by}{\mathbf{y}}

\newcommand{\G}{\mathcal{G} }

\newcommand{\gn}[1]{\!^{{}^{{#1}}}\!\!}

\newcommand{\Z}{\mathbb{Z}}

\begin{document}

\title{Cellular automata can really solve the parity problem}

\date{\today}

\author[1,4]{\fnm{Barbara} \sur{Wolnik}}
\author*[2]{\fnm{Anna} \sur{Nenca}}\email{ anna.nenca@ug.edu.pl}
\author[3]{\fnm{Pedro Paulo} \sur{Balbi}}
\author[4]{\fnm{Bernard}  \sur{De Baets}}
\affil[1]{Institute of Mathematics, Faculty of Mathematics, Physics and Informatics, University of Gda\'{n}sk, 80-308 Gda\'{n}sk, Poland}
\affil[2]{Institute of Informatics, Faculty of Mathematics, Physics and Informatics, University of Gda\'{n}sk, 80-308 Gda\'{n}sk, Poland}
\affil[3]{Faculdade de Computa\c{c}\~{a}o e Inform\'{a}tica, Universidade Presbiteriana Mackenzie, 
Rua da Consola\c{c}\~{a}o 896, Consola\c{c}\~{a}o; 01302-907 S\~{a}o Paulo, SP, Brazil}
\affil[4]{KERMIT, Department of Data Analysis and Mathematical Modelling, Faculty of Bioscience Engineering, Ghent University, Coupure links 653, B-9000 Gent, Belgium}

\abstract{
Determining properties of an arbitrary binary sequence is a challenging task if only local processing is allowed. Among these properties, the determination of the parity of 1s by distributed consensus has been a recurring endeavour in the context of automata networks. In its most standard formulation, a one-dimensional cellular automaton rule should process any odd-sized cyclic configuration and lead the lattice to converge to the homogeneous fixed point of 0s if the parity of 1s is even and to the homogeneous fixed point of 1s, otherwise. The only known solution to this problem with a single rule was given by Betel, de Oliveira and Flocchini
(coined BFO rule after the authors' initials). However, three years later the authors of the BFO rule realised that the rule would fail for some specific configuration and proposed a computationally sound fix, but a proof could not be worked out. Here we provide another fix to the BFO rule along with a full proof, therefore reassuring that a single-rule solution to the problem really does exist.
}


 \maketitle

\section{Introduction}

Determining properties of an arbitrary binary sequence is a challenging task if only local processing is allowed. Among these properties, the determination of the parity of 1s has been a recurring endeavour in the context of automata networks. In its most standard formulation, a one-dimensional cellular automaton rule should process any odd-sized cyclic configuration and lead the grid to converge to the homogeneous fixed point of 0s if the parity of 1s is even and to the homogeneous fixed point of 1s, otherwise.
Note that according to the formulation of the problem, it is immediately clear that it makes no sense for even-sized grids, as then the homogeneous configuration of 1s could not be a fixed point.

As such, the solution is achieved by a process of distributed consensus that is reached by all cells in the lattice.   
Many attempts to find a single rule that would solve this so-called \emph{parity problem} have been reported, the most successful one being~\cite{Wolz_Balbi_2008}, where an evolutionary search process allowed to find various extremely good, though not perfect rules. The only known solution to the problem with a single rule was given in~\cite{BFO2013}, coined BFO rule after the authors' alphabetically ordered initials. The BFO rule is a radius-4 one-dimensional cellular automaton, \emph{i.e.}, with its state transitions acting on nine neighbouring cells. It has also been shown that any solution to the classical parity problem cannot be too simple, as it requires a neighbourhood containing at least 7 cells~\cite{No6}.

Alternative formulations to tackle the problem, quite often and surprisingly with the much simpler cellular automata rules of the elementary space -- where neighbourhoods of only 3 cells are considered -- have appeared in the literature. Among them, it is worth mentioning the use of a non-uniform cellular automaton~\cite{Sipper_1998}; the employment of distinct rules at different moments along the time evolution~\cite{Lee_Chu_Xao_2001, Martin_Balbi_improvement_2009}; a rule (the local parity of the elementary space) updating the lattice in a deterministic asynchronous fashion~\cite{AsynchParityECA150_2019}; automata networks operating not on the regular connection pattern of the cells that characterise cellular automata but on a family of connection graphs~\cite{SynchParityECA150Graph_2022, GenSynchParityGraph_2024}; and the use of stochastic asynchronous updating~\cite{Fates_2024}. These   continued efforts testify that the parity problem remains theoretically relevant, since it provides a benchmark view on the possibilities and limitations of computing in a totally local distributed fashion.

Back to the aforementioned BFO solution, three years after its proposal, the authors of~\cite{BFO2013} realised that the rule would not work for the specific initial configuration 0001110101001 of size 13. The problem occurs because three tasks simultaneously perform joint transformations in the configuration that end up leading the configuration back to itself, with a spatial displacement, in a periodic regime.
As soon as the authors realised the issue they proposed a fix that led to a new rule they named BFOm rule. This rule was then computationally probed for all configurations up to size 31 and all the details of the fix were explained in detail in~\cite{BFOm2016}, which was kept unpublished because a proof could not be worked out. Here we provide another fix to the BFO rule along with a full proof, therefore reassuring that (at least) one single-rule solution to the problem really does exist, contrarily to a recent conjecture about this possibility~\cite{Celestine_2023}. Figure~\ref{fig:ts-diagram} illustrates the temporal evolution of the original BFO rule on the faulty initial configuration and the new one with the fixed version of BFO.

In the next section, we provide the necessary background for this paper (definitions and notations), together with results on the original BFO rule that remain valid. In Section~\ref{sec:dyn_switches} we consider the mechanisms underlying the temporal evolution associated with the alleged solution and amend the one that was failing in the original BFO rule. Section~\ref{sec:main_theorem} contains the proof that the new rule indeed solves the parity problem for any arbitrary initial configuration. In the last section we finish with some concluding remarks.

\section{Notation and main definitions}
We consider one-dimensional binary cellular automata on finite grids with periodic boundary conditions. For a positive integer $n$, the $n$-cell grid (or lattice) $\G_n$ refers to the grid of cells numbered by $0,1,\ldots,n-1$, from left to right, \emph{i.e.}, 
\[
\G_n =\{0,1,\dots,n-1\}\,,
\]
and we assume that the last (rightmost) cell $n-1$ is adjacent to the first cell $0$ (the cells are arranged along a circle) hence, all operations on indices are performed modulo~$n$ or, in other words, $\G_n =\Z/n\Z$. Any $n$-tuple $\bx =(x_0,x_1,\ldots,x_{n-1})\in X_n = \{0,1\}^n$ is called \emph{a configuration} (of length $n$) and represents the states of the cells in $\G_n$: the state of cell $i\in \G_n$ in a configuration $\bx$ is denoted by $x_i$. For a given $n$, there are two configurations that are called homogeneous: $(0,0,\ldots,0)$ and $(1,1,\ldots,1)$. 
For the sake of brevity, we will often express configurations in a more compact form and write them without commas. We will also drop the index $n$ and write $\G$ instead of $\G_n$, whenever the context is clear.

To consider the parity problem, we will not limit ourselves to a single grid, but we will consider an entire family of grids at the same time. For this purpose, we consider the following set of configurations:
\[
\xodd=\bigcup_{n=1}^\infty X_{2n-1} = \bigcup_{n=1}^\infty \{0,1\}^{2n-1}\, ,
\]
\emph{i.e.}, the set of all configurations of all possible odd lengths.
The \emph{parity} of a configuration $\bx\in X_n$ is denoted by 
$\mathrm{Par}(\bx)$, \emph{i.e.},
\[
\mathrm{Par}(\bx) = x_0\oplus x_1\oplus x_2\oplus \ldots\oplus x_{n-1}\, .
\]
In other words, $\mathrm{Par}(\bx)$ equals $0$ or $1$, depending on whether there is an even or an odd number of $1$s in the configuration $\bx$.

In each subsequent time step, each cell $i\in\G$ updates its state based on its current state and the states of its $r$ left and $r$ right neighbours, where $r \in \{0,1,2,\ldots\}$. 
Formally speaking, there is a local rule $f:\{0,1\}^{2r+1}\to \{0,1\}$, which generates the global rule $F : \xodd \to \xodd$ in the following way:
\[
F(\bx)_i = f(x_{i-r},x_{i-r+1}, \dots, x_{i+r-1}, x_{i+r})\,,
\]
for any $\bx\in \xodd$.
Further, $F^t(\bx)$ denotes the result of the $t$-th application of~$F$ to the configuration~$\bx$, corresponding to the $t$-th time step.

\begin{defn}
A local rule $f$ is said to solve the parity problem if for each  $\bx\in \xodd$ there exists a time step $t_0\geq 0$ such that for all $t\geq t_0$ the configuration $F^t(\bx)$ is homogeneous and consists either of $1$s if $\mathrm{Par}(\bx)=1$, or consists of $0$s if $\mathrm{Par}(\bx)=0$.
\end{defn}

This paper focuses on the proof that there exists such a local rule that can solve the parity problem. In fact, such a rule has been described in~\cite{BFO2013} and is referred to as the BFO rule. This local rule has radius 4 (\emph{i.e.}, its neighbourhood consists of nine cells) and, in general, the idea underlying the BFO rule design is to reduce the number of blocks of 0s and 1s present in the initial configuration. This is accomplished by propagating blocks of 1s to the right side, two cells per iteration, combined with the simultaneous propagation of 0s to the left. Actually, the BFO rule is much more complicated, as it must include appropriate stopping conditions for these two major trends to coexist. 

However, in its original formulation in~\cite{BFO2013} a slight inaccuracy was present that can easily be corrected; namely, the sentence ``{\it Local shift}: a $(101)$ block is transformed into $(110)$ if there are a 0 on its left and at least two 0s on its right (combination of transitions $T_7$ and $T_8$)" should be in fact ``{\it Local shift}: a $(101)$ block is transformed into $(011)$ if there are at least two 0s on its left and a 0 on its right (combination of transitions $T_7$ and $T_8$)". As will become clear below, $T_7$ and $T_8$ simply refer to compact representations of a subset of the \emph{active state transitions} (ATs) of the rule, which are those associated with a state change in the central cell of the neighbourhood.
\begin{figure}[h]
   \[
    \xymatrix@C=0.2cm@R=-0.03cm{ 
 t=0 &  0001110101001 & 0001110101001 \\
t=1 &  1101010001001 & 1101010010001\\
t=2 & 0100010001111 & 0100100011101 \\
t=3 &  0000011101101 & 1000111010100\\
t=4 & 0000010000001 & 1110101001000\\
t=5 & 1100011100001 & 1010010001110\\
t=6 &  1111011111001 & 0100011101010\\
t=7 & 1101001111111 & 0111010100100\\
t=8 & 0100001111111 & 0101001000111\\
t=9 &  0000001111101 & 0010001110101\\
t=10 &  0000001110100 & 0011101010010\\
t=11 & 0000001010000 & 1010100100011\\
t=12 &  0000000110000 & 1001000111010\\
t=13 &  0000000000000 & 0001110101001
}\]
    \caption{The space-time diagram for the configuration $\bx=0001110101001$ mentioned in the introduction, for the corrected version of the BFO rule on the left and the original BFO rule on the right.}
    \label{fig:ts-diagram}
\end{figure}

Consequently, $T_7$ and $T_8$ had to be properly redefined, which is achieved by simply mirror-reflecting them. Such a corrected version of the BFO rule is the subject of this paper; and although the correction is quite simple, the proof of its correctness ended up leading to a completely new proof technique. 
Since the state transition table of the BFO rule contains 512 values, it is not necessary to specify it in full, but only the active state transitions, as indicated in Table~\ref{tab}, already with the corrected versions of $T_7$ and $T_8$. In terms of its representation in the space of binary cellular automata with radius 4 (neighbourhoods with 9 cells), its rule number is: 12766019579927887748828308653663109277301603915220967933337785052737964273\\3523952685217154493686311891
4126592211732878316055036275868139520320981134\\1541376 (following Wolfram's lexicographic ordering of the neighbourhoods).

\begin{table}[h]
	\centering
	\begin{tabular}{c}
 \xymatrix@R=0,5cm{ 
T_1 :[\bullet111{\bf{0}}0{\bullet}{\bullet}{\bullet}]  & T_5:[{\bullet}{\bullet}{\bullet}0{\textbf{1}}10{\bullet}{\bullet}]    \\
T_2:[1110{\textbf{0}}{\bullet}{\bullet}{\bullet}{\bullet}] & T_6: [{\bullet}{\bullet}01{\textbf{1}}0{\bullet}{\bullet}{\bullet}]    \\
T_3:[{\bullet}001{\textbf{0}}0{\bullet}{\bullet}{\bullet}]  &  T_8: [{\bullet}{\bullet}00{\textbf{1}}010{\bullet}]       \\
T_4:[0010{\textbf{0}}{\bullet}{\bullet}{\bullet}{\bullet}] &  T_9: [{\bullet}{\bullet}{\bullet}1{\textbf{1}}101{\bullet}]   \\
T_7:[{\bullet}001{\textbf{0}}10{\bullet}{\bullet}] & T_{10}:[1110{\textbf{1}}0{\bullet}{\bullet}{\bullet}]  \\
& T_{11}: [1110{\textbf{1}}11{\bullet}{\bullet}]\\
& T_{12}: [{\bullet}{\bullet}11{\textbf{1}}0110]
 }
\end{tabular}
\caption{The state transition table of the BFO rule, from the viewpoint of its active transitions; the central position in the neighbourhood is indicated in boldface and the symbol $\bullet$ refers to \emph{don't care} positions.}\label{tab}
\end{table}

The explicit blocks in the ATs allow to define seven pairs of ATs, each one acting on a block that leads to a corresponding image, as shown in Table~\ref{SevenATpairs} (with both kinds of blocks represented between vertical dashes).

\begin{table}[h]
	\centering
	\begin{tabular}{c}
{\scriptsize

\xymatrix@C=0.3cm@R=0.5cm{ 
T_{1,2} & T_{3,4} & T_{5,6} & T_{7,8} & T_{9,10} & T_{9,11} & T_{9,12} \\ 
|\textcolor{gray}{ 111}\bf{00}| \ar[d] & |\textcolor{gray}{ 001}\bf{00}| \ar[d] & |\textcolor{gray}{ 0}{\bf{11}}\textcolor{gray}{ 0}| \ar[d] & |\textcolor{gray}{ 00}{\bf{10}}\textcolor{gray}{ 10}|\ar[d]  &           |\textcolor{gray}{ 1}{\bf{1}}\textcolor{gray}{ 10}{\bf{1}}\textcolor{gray}{ 0}|\ar[d] & |\textcolor{gray}{ 1}{\bf{1}}\textcolor{gray}{ 10}{\bf{1}}\textcolor{gray}{ 11}|\ar[d]  & |\textcolor{gray}{ 1}{\bf{11}}\textcolor{gray}{ 0110}|\ar[d]   \\
|\bb\bb\bb{\bf11}| & |\bb\bb\bb{\bf11}|          &   |\bb{\bf00}\bb|          & |\bb\bb{\bf01}\bb\bb|     & |\bb{\bf0}\bb\bb{\bf0}\bb| & |\bb{\bf0}\bb\bb{\bf0}\bb\bb| & |\bb{\bf00}\bb\bb\bb\bb| 
}
}

\end{tabular}
\caption{The seven AT pairs.}\label{SevenATpairs}
\end{table}

We refer to each of the blocks at the top as the domain $D$ of the corresponding AT pair, and denote them by $D_{1,2}$, $D_{3,4}$ and so on; analogously, we refer to each of the blocks at the bottom as the image $V$ of the corresponding AT pair, and denote them by $V_{1,2}$, $V_{3,4}$ and so on. With the use of Table~\ref{tab}, we can partially complete the image blocks as shown in Table~\ref{tab-2}.

\begin{table}[h]
	\centering
	\begin{tabular}{c}
{\footnotesize
\xymatrix@C=0.3cm@R=0.5cm{ 
D_{1,2} & D_{3,4} & D_{5,6} & D_{7,8} & D_{9,10} & D_{9,11} & D_{9,12} \\ \vspace{3cm}
|\textcolor{gray}{ 111}{\bf{00}}| \ar[d] & |\textcolor{gray}{ 001}{\bf{00}}| \ar[d] & |\textcolor{gray}{ 0}{\bf{11}}\textcolor{gray}{ 0}| \ar[d] & |\textcolor{gray}{ 00}{\bf{10}}\textcolor{gray}{ 10}|\ar[d]  &           |\textcolor{gray}{ 1}{\bf{1}}\textcolor{gray}{ 10}{\bf{1}}\textcolor{gray}{ 0}|\ar[d] & |\textcolor{gray}{ 1}{\bf{1}}\textcolor{gray}{ 10}{\bf{1}}\textcolor{gray}{ 11}|\ar[d]  & |\textcolor{gray}{ 1}{\bf{11}}\textcolor{gray}{ 0110}|\ar[d]   \\
|\textcolor{gray}{?11}{\bf{11}}| & |\textcolor{gray}{??1}{\bf{11}}|          &   |\textcolor{gray}{?}{\bf{00}}\textcolor{gray}{0}|          & |\textcolor{gray}{??}{\bf{01}}\textcolor{gray}{10}|     & |\textcolor{gray}{?}{\bf{0}}\textcolor{gray}{10}{\bf{0}}\textcolor{gray}{0}| & |\textcolor{gray}{?}{\bf{0}}\textcolor{gray}{10}{\bf{0}}\textcolor{gray}{??}| & |\textcolor{gray}?{\bf{00}}\textcolor{gray}{0000}| \\
V_{1,2} & V_{3,4} & V_{5,6} & V_{7,8} & V_{9,10} & V_{9,11} & V_{9,12}
}}
\end{tabular}
\caption{The domains and partially completed images of all seven AT pairs.}\label{tab-2}
\end{table}

Generally, in place of every question mark in $V$ the same value as above should appear, unless some domains overlap. Since by convention we view each configuration from left to right, then we say that, for example, in the configuration $\bx=111001010$ the domain $D_{1,2}$ is overlapped by the domain $D_{7,8}$ (or that $D_{7,8}$ overlaps $D_{1,2}$).

In~\cite{BFO2013} it has been shown that the following theorems hold (the correction of the BFO has no influence on these results).

\begin{thm}\label{thm:BFO1}
    The BFO rule preserves the parity of the configuration.
\end{thm}

\begin{thm}\label{thm:BFO2}
    The only fixed points of the BFO rule are the homogeneous configurations.
\end{thm}

\noindent
The latter theorem means that if $\bx\in\xodd$ is not homogeneous, then at least one AT pair will act during the next update of $\bx$.

\bigskip

To analyse the dynamics of the BFO rule, we define `box’ and `switch’ patterns that quantify the non-homogeneity of a configuration.

\begin{defn}
    Let $\bx\in\xodd$. We say that $x_ix_{i+1}$ is a box if $x_ix_{i+1}=01$, $x_{i-1}=1$ and $x_{i+2}x_{i+3}=00$.
\end{defn}

\noindent In other words, a box is the pattern `01' preceded by 1 and followed by 00. 

In a given configuration there might be one or more boxes, but there might also be no box at all. Below we give some sample configurations with all boxes (if they appear) marked with gray rectangles:
\[
1110101\fcolorbox{gray}{white}{01}000111,\quad 111\fcolorbox{gray}{white}{01}001\fcolorbox{gray}{white}{01}00111,\quad 111011100001111.
\]

\begin{defn}\label{def:switch}
Let $\bx\in\xodd$. We say that $\bx$ has a b-switch (block switch) at position $i$ if $x_{i+1}x_{i+2}$ is a box. We say that $\bx$ has an r-switch (regular switch) at position $i$ if $x_i\neq x_{i+1}$ and both $x_i$ and $x_{i+1}$ do not belong to any box. By $s(\bx)$ we denote the number of all switches in~$\bx$.
\end{defn}

 In other words, to express how different a given configuration is from a homogeneous one, we will use a slightly modified measure compared to the one introduced in~\cite{BFO2013}. More explicitly, instead of counting all places where 0 changes to 1 (or vice versa) in the entire configuration, we propose to first extract some parts of the configuration (all boxes) and only then count such changes in the retained parts.
The following are sample configurations with all switches numbered, and the boxes identified by gray rectangles. Note that if a given switch is at position $i$, its identifying number is placed in between the cells $i$ and $i+1$, at the top:
\[
111\gn{1}0\gn{2}1\gn{3}0\gn{4}1\gn{5}\fcolorbox{gray}{white}{01}000\gn{6}111,\quad 111\gn{1}\fcolorbox{gray}{white}{01}00\gn{2}1\gn{3}\fcolorbox{gray}{white}{01}00\gn{4}111,\quad 111\gn{1}0\gn{2}111\gn{3}0000\gn{4}1111.
\]

The following fact will turn out to be very important.
\begin{prop}
\label{prop:hom=0}
A configuration $\bx\in\xodd$ is a homogeneous configuration if and only if $s(\bx)=0$.   
\end{prop}
\begin{proof}
    Let $\bx = x_0x_1\ldots x_{n-1}\in \xodd$. Of course, if $\bx$ is a homogeneous configuration, then $s(\bx)=0$. Now, suppose that neither $\bx = 000\ldots00$ nor $\bx = 111\ldots 11$. This means that there exists at least one $i\in\{0,1,\ldots,n-1\}$ such that $x_i=0$ and $x_{i+1}=1$. If $x_ix_{i+1}=01$ is a box, then there is a b-switch at position $n-1$, and so $s(\bx)\geq 1$. If $x_ix_{i+1}=01$ is not a box, then, according to Definition~\ref{def:switch}, there is an r-switch at position $i$, and therefore also in this case $s(\bx)\geq 1$.
\end{proof}

Figure~\ref{fig:example} presents a sample configuration $\bx$ in which all switches have been numbered. This allows us to track how they are affected by the action of the BFO rule.
One should note that, as time evolves, the switches from the initial configuration subsequently move to the right or to the left, or simply disappear; but no new ones appear.
As a result, after 27 updates, all switches disappear and $F^{27}(\bx)$ is a homogeneous configuration.
This is reflected in the sequence $\big(s(F^t(\bx))\big)_{t=0}^\infty$, which is monotonically decreasing, reaching zero for $t\geq 27$. As will become apparent later in the paper, the same scenario unfolds for any configuration $\bx\in \xodd$.

\begin{figure}
   \[
    \xymatrix@C=0.2cm@R=-0.03cm{ 
 t=0 &   00000\gn{1}1\gn{2}0\gn{3}111\gn{4}00\gn{5}1\gn{6}0\gn{7}11111\gn{8} & s(\bx)=8 \\
t=1 & 11\gn{8}000\gn{1}1\gn{2}0\gn{3}111111\gn{6}0\gn{7}11111 & s(F(\bx))=6 \\
t=2 & 1111\gn{8}0\gn{1}1\gn{2}0\gn{3}1111\gn{6}\fcolorbox{gray}{white}{01}00\gn{7}1111 & s(F^2(\bx))=6 \\
t=3 & 11\gn{8}\fcolorbox{gray}{white}{01}000\gn{3}11\gn{6}\fcolorbox{gray}{white}{01}0000\gn{7}1111 & s(F^3(\bx))=4 \\
t=4 & \gn{8}\fcolorbox{gray}{white}{01}00000000\gn{3}1\gn{6}0000\gn{7}1111 & s(F^4(\bx))=4 \\
t=5 & 0000000000\gn{3}111\gn{6}00\gn{7}11\gn{8}\fcolorbox{gray}{white}{01} & s(F^5(\bx))=4 \\
t=6 & 0000000000\gn{3}11111\gn{6}000\gn{7}1\gn{8} & s(F^6(\bx))=4 \\
t=7 & 11\gn{8}00000000\gn{3}1111111\gn{6}0\gn{7}1 & s(F^7(\bx))=4 \\
t=8 & 1111\gn{8}000000\gn{3}11111\gn{6}\fcolorbox{gray}{white}{01}00\gn{7} & s(F^8(\bx))=4 \\
t=9 & 111111\gn{8}0000\gn{3}111\gn{6}\fcolorbox{gray}{white}{01}0000\gn{7} & s(F^9(\bx))=4 \\
t=10 & 11111111\gn{8}00\gn{3}1\gn{6}\fcolorbox{gray}{white}{01}000000\gn{7} & s(F^{10}(\bx))=4 \\
t=11 & 1111111111\gn{8}0\gn{3}11\gn{6}000000\gn{7} & s(F^{11}(\bx))=4 \\
t=12 & 11111111\gn{8}00000000000\gn{7} & s(F^{12}(\bx))=2 \\
t=13 & 1111111111\gn{8}000000000\gn{7} & s(F^{13}(\bx))=2 \\
t=14 & 111111111111\gn{8}0000000\gn{7} & s(F^{14}(\bx))=2 \\
t=15 & 11111111111111\gn{8}00000\gn{7} & s(F^{15}(\bx))=2 \\
t=16 & 1111111111111111\gn{8}000\gn{7} & s(F^{16}(\bx))=2 \\
t=17 & 111111111111111111\gn{8}0\gn{7} & s(F^{17}(\bx))=2 \\
t=18 & 0\gn{7}111111111111111\gn{8}\fcolorbox{gray}{white}{01}0 & s(F^{18}(\bx))=2 \\
t=19 & 0\gn{7}1111111111111\gn{8}\fcolorbox{gray}{white}{01}000 & s(F^{19}(\bx))=2 \\
t=20 & 0\gn{7}11111111111\gn{8}\fcolorbox{gray}{white}{01}00000 & s(F^{20}(\bx))=2 \\
t=21 & 0\gn{7}111111111\gn{8}\fcolorbox{gray}{white}{01}0000000 & s(F^{21}(\bx))=2 \\
t=22 & 0\gn{7}1111111\gn{8}\fcolorbox{gray}{white}{01}000000000 & s(F^{22}(\bx))=2 \\
t=23 & 0\gn{7}11111\gn{8}\fcolorbox{gray}{white}{01}00000000000 & s(F^{23}(\bx))=2 \\
t=24 & 0\gn{7}111\gn{8}\fcolorbox{gray}{white}{01}0000000000000 & s(F^{24}(\bx))=2 \\
t=25 & 0\gn{7}1\gn{8}\fcolorbox{gray}{white}{01}000000000000000 & s(F^{25}(\bx))=2 \\
t=26 & 00\gn{7}11\gn{8}000000000000000 & s(F^{26}(\bx))=2 \\
t=27 & 0000000000000000000 & s(F^{27}(\bx))=0 
}\]
    \caption{The space-time diagram for the sample configuration $\bx=0000010111001011111$ with all switches numbered.}
    \label{fig:example}
\end{figure}

\section{The dynamics of switches}\label{sec:dyn_switches}
In this section we study how every particular AT pair affects the total number of switches. Our goal is to show that the number of switches is non-increasing, \emph{i.e.}, that for any configuration $\bx\in\xodd$ the sequence $\big(s(F^t(\bx))\big)_{t=0}^\infty$ monotonically decreases. 

In the following subsections, we compare the switches in every domain $D$ and corresponding image $V$ in detail to justify that every switch present in $V$ is actually a switch from $D$, which may have moved to the left or to the right (in order to make it easier to track the movement of the switches, they are identified with markers from the set of symbols \{$*$, $\circ$, $\prime$, $\diamond$, $\flat$\}, also placed in between the corresponding cells, at the top).
This task is complicated by the fact that sometimes we cannot resolve what type of switch is in $D$, or some domains may overlap, all of which require consideration of cases. 

In what follows, we assume that $\bx\in\xodd$ and $\by=F(\bx)$, and that all values given in $V$ are calculated based on Table~\ref{tab}. Furthermore, it is worth keeping in mind that there is no need to be concerned about short configurations, since if $\bx$ is a short configuration, then we can consider a long configuration $\bx^k=\bx\bx\ldots\bx$, which is a concatenation of as many $\bx$'s as necessary. So, due to periodic boundary conditions, for each time step $t$, we have $F^t(\bx^k)=(F^t(\bx))^k$; hence, if we know the dynamics of the global rule $F$ for long configurations, we also know it for short ones.

\subsection{The pairs $T_{1,2}$ and $T_{3,4}$}
The pair $T_{1,2}$ acts always when the configuration $\bx$ contains $D_{1,2}=|11100|$. Without loss of generality, we can assume that $|x_1x_{2}x_{3}x_{4}x_{5}|=|11100|$. Note that the pair $T_{1,2}$ changes $x_{4}x_{5}=00$ to $y_{4}y_{5}=11$, thus it can only affect switches at positions $3$ and $5$; therefore, we do not need to care about the question mark at position $y_1$.
If $y_6=0$, then $V_{1,2}$ has a switch at position 5, but it is the switch from position 3 in $D_{1,2}$ that 
has moved two cells to the right (Figure~\ref{fig:t1t2(a)}). If $y_6=1$, then $V_{1,2}$ neither has a switch at position 3 nor at position 5, which means that both switches from $D_{1,2}$ disappear, since two blocks of 1s merge (Figures~\ref{fig:t1t2(b)}). 
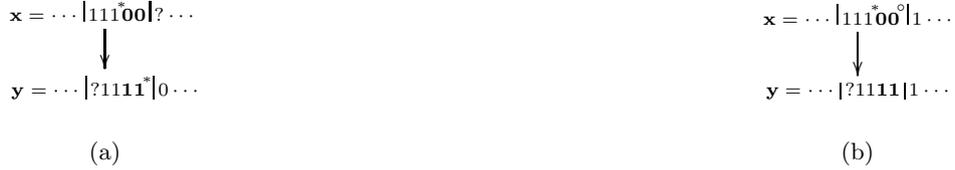
\begin{figure}[h!]
    \centering
    \begin{subfigure}{0.45\textwidth}
{\scriptsize\[
\xymatrix@C=0.2cm@R=0.5cm{ 
\bx = \cdots \;\lllp 111\gn{*}\boldsymbol{00} \, \lllp\, ?\cdots\;\ar[d] \\
\by = \cdots \;\lllp\, ?11\boldsymbol{11}\gn{*} \;\lllp\, 0\cdots
}
\]}
    \caption{}
    \label{fig:t1t2(a)}
\end{subfigure}
\hfill
 \begin{subfigure}{0.45\textwidth}
{\scriptsize\[
\xymatrix@C=0.2cm@R=0.5cm{ 
\bx = \cdots \;\lllp  111\gn{*}\boldsymbol{00}\gn{\circ} \;\lllp\,1 \cdots\ar[d] \\
\by = \cdots \;\lllp \, ?11\boldsymbol{11} \,\lllp\, 1\cdots 
}
\]}
    \caption{}
    \label{fig:t1t2(b)}
\end{subfigure}
 \caption{The effect of $T_{1,2}$ on switches in $D_{1,2}$ in the case (a) $y_{6}=0$, (b) $y_{6}=1$.}
        \label{fig:t1t2}
\end{figure}

We can recognize both situations in Figure~\ref{fig:example} when going from $t=0$ to $t=1$: switch 4 disappears, while switch 8 moves two cells to the right. We summarize our analysis below.

\begin{lem}\label{lem:t1t2}
The action of the pair $T_{1,2}$ does not produce new switches. Moreover, if as a result of its action two blocks of 1s merge, then this leads to a reduction in the number of switches.
\end{lem}

\bigskip

Analogous considerations for the pair $T_{3,4}$ are presented in Figure~\ref{fig:t3t4}.

\begin{figure}[h!]
    \centering
    \begin{subfigure}{0.45\textwidth}
{\scriptsize\[
\xymatrix@C=0.2cm@R=0.5cm{ 
\bx = \cdots \;\lllp 001\gn{*}\boldsymbol{00} \, \lllp\,? \cdots\;\ar[d] \\
\by = \cdots \;\lllp\,??1\boldsymbol{11}\gn{*} \;\lllp\,0 \cdots
}
\]}
    \caption{}
    \label{fig:t3t4(a)}
\end{subfigure}
\hfill
 \begin{subfigure}{0.45\textwidth}
{\scriptsize\[
\xymatrix@C=0.2cm@R=0.5cm{ 
\bx = \cdots \;\lllp  001\gn{*}\boldsymbol{00}\gn{\circ} \;\lllp\,1 \cdots\ar[d] \\
\by = \cdots \;\lllp \,??1\boldsymbol{11} \,\lllp\, 1\cdots 
}
\]}
    \caption{}
    \label{fig:t3t4(b)}
\end{subfigure}
 \caption{The effect of $T_{3,4}$ on switches in $D_{3,4}$ in the case (a) $y_{6}=0$, (b) $y_{6}=1$.}
        \label{fig:t3t4}
\end{figure}
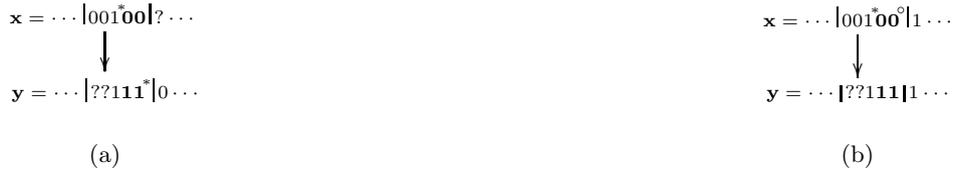

As a result we get the following conclusion.

\begin{lem}\label{lem:t3t4}
The action of the pair $T_{3,4}$ does not produce new switches. Moreover, if as a result of its action two blocks of 1s merge, then this leads to a reduction in the number of switches.
\end{lem}

\subsection{The pair $T_{5,6}$} 

Assume that $|x_1x_{2}x_{3}x_{4}|=|0110|=D_{5,6}$. Since the pair $T_{5,6}$ only changes $x_{2}x_{3}=11$ to $y_{2}y_{3}=00$, it can only affect switches at positions 1 and 3. However, we should be careful, because the switch at position 3 can be a b-switch, in which case its coverage is wider. For this reason, it is convenient to consider two variants of the domain $D_{5,6}$: $D_{5,6}^r=|0110|$ with $x_5x_6x_7\neq 100$, and $D_{5,6}^b=|0110100|$. In $D_{5,6}^r$ there are two r-switches, while $D_{5,6}^b$ contains one r-switch and one b-switch. 

Figure~\ref{fig:t5t6} shows the effect of the pair $T_{5,6}$ on the switches when assuming that $y_1=0$.
In the case of $D_{5,6}^r$, both r-switches at positions $1$ and $3$ disappear (Figure~\ref{fig:t5t6(a)}).
In the case of $D_{5,6}^b$, both the r-switch at position $1$ and the b-switch at position $4$ move to the right; 
moreover, the b-switch turns into an r-switch (Figure~\ref{fig:t5t6(b)}).
\begin{figure}[h!b]
    \centering
    \begin{subfigure}{0.46\textwidth}
{\scriptsize\[
\xymatrix@C=0.2cm@R=0.5cm{ 
\bx = \cdots \;\lllp  0\gn{*}\boldsymbol{11}\gn{\circ} \ar[d]  0\,\lllp\,\cdots\ar[d]  \\
\by = \cdots \;\lllp 0\boldsymbol{00} 0\,\lllp\,\cdots 
}
\]}
    \caption{}
    \label{fig:t5t6(a)}
\end{subfigure}
\hfill
\begin{subfigure}{0.45\textwidth}
{\scriptsize\[
\xymatrix@C=0.2cm@R=0.5cm{ 
\bx = \cdots \;\lllp  0\gn{*}\boldsymbol{11}\gn{\circ} \ar[d]  \fcolorbox{gray}{white}{01}00\,\lllp\,\cdots\ar[d]  \\
\by = \cdots \;\lllp 0\boldsymbol{00}0   \gn{*}1\gn{\circ}00\,\lllp\,\cdots 
}
\]}
    \caption{}
    \label{fig:t5t6(b)} 
\end{subfigure}
\caption{The effect of $T_{5,6}$ on the switches in the case $y_1=0$ and (a) $D_{5,6}^r$, (b) $D_{5,6}^b$.}\label{fig:t5t6}
\end{figure}
\begin{figure}[h!b]
    \centering
    \begin{subfigure}{0.46\textwidth}
{\scriptsize\[
\xymatrix@C=0.2cm@R=0.5cm{ 
\bx = \cdots 111\gn{\diamond}0\,\lllp \, 0\gn{*}\boldsymbol{11}\gn{\circ} \ar[d]  0\,\lllp\,\cdots\ar[d]  \\
\by = \cdots \,?111\,\lllp\,1\gn{\diamond}\boldsymbol{00} 0\,\lllp\,\cdots 
}
\]}
    \caption{}
    \label{fig:t5t6-cd(a)}
\end{subfigure}
\hfill
\begin{subfigure}{0.45\textwidth}
{\scriptsize\[
\xymatrix@C=0.2cm@R=0.5cm{ 
\bx = \cdots 111\gn{\diamond}0\,\lllp \, 0\gn{*}\boldsymbol{11}\gn{\circ} \ar[d]  \fcolorbox{gray}{white}{01}00\,\lllp\,\cdots\ar[d]  \\
\by = \cdots \, ?111\,\lllp\,1\gn{\diamond}\boldsymbol{00}0   \gn{*}1\gn{\circ}00\,\lllp\,\cdots 
}
\]}
    \caption{}
    \label{fig:t5t6-cd(b)} 
\end{subfigure}
\caption{The effect of $T_{5,6}$ on the switches in the case $y_1=1$ and (a) $D_{5,6}^r$, (b) $D_{5,6}^b$.}\label{fig:t5t6-cd}
\end{figure}

We can also recognize both situations in Figure~\ref{fig:example}: switches 3 and 6 move to the right at time step 4, 
and disappear at time step 12.

Next, assume that $y_1=1$, which must have resulted from $T_2$ or $T_4$ (Table~\ref{tab}); so, as described in the previous subsection, the switch at position 1 in $V_{5,6}$ is in fact the switch from $D_{1,2}$ or $D_{3,4}$, respectively, overlapped by the considered $D_{5,6}$, which has moved two cells to the right. Figure~\ref{fig:t5t6-cd} depicts the situation in the case of $T_2$.

We summarize our analysis below.

\begin{lem}\label{lem:t5t6}
The action of the pair $T_{5,6}$ does not produce new switches. Moreover, if $T_{5,6}$ acts on $D_{5,6}^r$, then it leads to a reduction in the number of switches.
\end{lem}

\subsection{The pair $T_{7,8}$} 

Let $|x_1x_{2}x_{3}x_{4}x_5x_6|=|001010|=D_{7,8}$. Since the pair $T_{7,8}$ only changes $x_{3}x_{4}=10$ to $y_{3}y_{4}=01$, it can only affect switches at positions 2, 3 and 4. Thus, there is no need to consider the switches at positions 1 and 5.
As before, it is convenient to distinguish two variants of the domain $D_{7,8}$: $D_{7,8}^r=|0010101|$ and $D_{7,8}^b=|0010100|$.

Let us first assume that $y_2=0$. In the case of $D_{7,8}^r=|0010101|$, all three switches at positions 2, 3, 4 are regular; as shown in Figure~\ref{fig:t7t8(a)}, the action of $T_{7,8}$ on $D_{7,8}^r$ reduces the number of these switches. In turn, $D_{7,8}^b=|0010100|$ has an r-switch at position 2 and a b-switch at position 3, which in the next time step both move to the right while the b-switch turns into an r-switch. Moreover, as a result we obtain $D_{5,6}^b$ in $\by$ (Figure~\ref{fig:t7t8(b)}).
\begin{figure}[h!b]
    \centering
    \begin{subfigure}{0.45\textwidth}
{\scriptsize\[
\xymatrix@C=0.2cm@R=0.5cm{ 
\bx = \cdots \;\lllp 00\gn{*}\boldsymbol{1}\gn{\circ}\boldsymbol{0}\gn{\prime} \ar[d]  101\,\lllp\,\cdots\ar[d]  \\
\by = \cdots \;\lllp{?}0\boldsymbol{0}\gn{*}\boldsymbol{1}10{?}\,  \lllp\,\cdots   
}
\]}
    \caption{}
    \label{fig:t7t8(a)}
\end{subfigure}
\hfill
 \begin{subfigure}{0.45\textwidth}
{\scriptsize\[
\xymatrix@C=0.2cm@R=0.5cm{ 
\bx = \cdots \;\lllp  00\gn{*}\boldsymbol{1}\gn{\circ} \fcolorbox{gray}{white}{{\bf{0}}1} 00\,\lllp\,\cdots \ar[d]   \\
\by = \cdots \;\lllp{?}0\boldsymbol{0}\gn{*}\boldsymbol{1}1\gn{\circ}00  \,\lllp\,\cdots  
}
\]}
    \caption{}
    \label{fig:t7t8(b)}
\end{subfigure}
\caption{The effect of $T_{7,8}$ on the switches in the case $y_2=0$ and (a) $D_{7,8}^r$, (b) $D_{7,8}^b$.}
\end{figure}

\begin{figure}[h!b]
    \centering
    \begin{subfigure}{0.45\textwidth}
{\scriptsize\[
\xymatrix@C=0.2cm@R=0.5cm{ 
\bx = \cdots 111\gn{\diamond}\;\lllp \, 00\gn{*}\boldsymbol{1}\gn{\circ}\boldsymbol{0}\gn{\prime} \ar[d]  101\,\lllp\,\cdots\ar[d]  \\
\by = \cdots {?}11 \;\lllp \,11\gn{\diamond}\boldsymbol{0}\gn{*}\boldsymbol{1}10?  \,\lllp\,\cdots   
}
\]}
    \caption{}
    \label{fig:t7t8(a)-cd}
\end{subfigure}
\hfill
 \begin{subfigure}{0.45\textwidth}
{\scriptsize\[
\xymatrix@C=0.2cm@R=0.5cm{ 
\bx = \cdots 111\gn{\diamond}\;\lllp \, 00\gn{*}\boldsymbol{1}\gn{\circ} \fcolorbox{gray}{white}{{\bf{0}}1} 00\,\lllp\,\cdots \ar[d]   \\
\by = \cdots {?}11\,\lllp\,11\gn{\diamond}\boldsymbol{0}\gn{*}\boldsymbol{1}1\gn{\circ}00  \,\lllp\,\cdots  
}
\]}
    \caption{}
    \label{fig:t7t8(b)-cd}
\end{subfigure}
\caption{The effect of $T_{7,8}$ on the switches in the case $y_2=1$ and  
        (a) $D_{7,8}^r$, (b) $D_{7,8}^b$.}
\label{fig:t7t8-cd}
\end{figure}

Next, assume that $y_2=1$, then it must be the result of $T_2$ or $T_4$, so, the switch at position 2 in $V_{5,6}$ is in fact the switch from $D_{1,2}$  or $D_{3,4}$ (overlapped by the considered $D_{7,8}$), which moved two cells to the right. We present such a situation in Figure~\ref{fig:t7t8-cd} in the case of $T_2$.

We summarize our analysis below.

\begin{lem}\label{lem:t7t8}
The action of the pair $T_{7,8}$ does not produce new switches. Moreover, if $T_{7,8}$ acts on $D_{7,8}^r$, then it leads to a reduction 
in the number of switches; if it acts on $D_{7,8}^b$, then $D_{5,6}^r$ is created in the next time step.
\end{lem}

Note that, due to Lemma~\ref{lem:t5t6}, the action of the pair $T_{7,8}$ always leads to a reduction in the number of switches, either immediately or after two time steps.

\subsection{The pair $T_{9,10}$} 
To compare the switches in $D_{9,10}=|x_1x_{2}x_{3}x_{4}x_{5}x_{6}|=|111010|$ and $V_{9,10}$, we need to consider three variants of this domain: $D_{9,10}^r=|1110101|$ with $x_8x_9\neq 00$, $D_{9,10}^b=|1110100|$, and $D_{9,10}^{rb}=|111010100|$.

\begin{figure}[h!]
    \centering
    \begin{subfigure}{0.32\textwidth}
{\scriptsize\[
\xymatrix@C=0.2cm@R=0.5cm{ 
\bx = \cdots \;\lllp 1{\boldsymbol{1}}1\gn{*}0\gn{\circ}{\boldsymbol{1}}\gn{\prime}   0\,\lllp\,\cdots\ar[d]  \\
\by = \cdots \;\lllp 1 \gn{*}\fcolorbox{gray}{white}{{\bf{0}}1}0\boldsymbol{0}0  \,\lllp\,\cdots   
}
\]}
    \caption{}
    \label{fig:t9t10(a)}
\end{subfigure}
\hfill
 \begin{subfigure}{0.32\textwidth}
{\scriptsize\[
\xymatrix@C=0.2cm@R=0.5cm{ 
\bx = \cdots \;\lllp1{\boldsymbol{1}}1 \ar[d]  \gn{*}\fcolorbox{gray}{white}{{0\bf{1}}} 00\,\lllp\,\cdots\ar[d]  \\
\by = \cdots \;\lllp 1\gn{*}\fcolorbox{gray}{white}{{\bf{0}}1} 0{\boldsymbol{0}}00  \,\lllp\,\cdots 
}
\]}
    \caption{}
    \label{fig:t9t10(b)}
\end{subfigure}
\hfill
\begin{subfigure}{0.32\textwidth}
{\scriptsize\[
\xymatrix@C=0.2cm@R=0.5cm{ 
\bx = \cdots \;\lllp  1{\boldsymbol{1}}1\gn{*}0\gn{\circ}\boldsymbol{1}   \gn{\prime}\fcolorbox{gray}{white}{01} 00\,\lllp\,\cdots\ar[d]    \\
\by = \cdots \;\lllp 1\gn{*}\fcolorbox{gray}{white}{{\bf{0}}1} 0{\boldsymbol{0}}0\gn{\circ}1\gn{\prime}00  \,\lllp\,\cdots  
}
\]}
    \caption{}
    \label{fig:t9t10(c)} 
\end{subfigure}
\caption{The effect of $T_{9,10}$ on the switches in the case $y_1=1$ and
        (a) $D_{9,10}^r$, (b) $D_{9,10}^b$, (c) $D_{9,10}^{rb}$.}
\end{figure}
\begin{figure}[!h]
    \centering
    \begin{subfigure}{0.32\textwidth}
{\scriptsize\[
\xymatrix@C=0.2cm@R=0.5cm{ 
\bx = \cdots 1110\gn{\diamond}\;\lllp  \,{1}\boldsymbol{1}1\gn{*}0\gn{\circ}\boldsymbol{1}\gn{\prime}   0\,\lllp\,\cdots\ar[d]  \\
\by = \cdots {?}010\;\lllp\;0{\bf{0}}\gn{\diamond}1\gn{*}0\boldsymbol{0}0  \,\lllp\,\cdots   
}
\]}
    \caption{}
    \label{fig:t9t10-cd(a)}
\end{subfigure}
\hfill
 \begin{subfigure}{0.32\textwidth}
{\scriptsize\[
\xymatrix@C=0.2cm@R=0.5cm{ 
\bx = \cdots 1110\gn{\diamond} \;\lllp \, 1\boldsymbol{1}1 \ar[d]  \gn{*}\fcolorbox{gray}{white}{{0\bf{1}}} 00\,\lllp\,\cdots\ar[d]  \\
\by = \cdots {?}010\;\lllp\,0{\bf{0}}\gn{\diamond}1 \gn{*}0\boldsymbol{0}00  \,\lllp\,\cdots 
}
\]}
    \caption{}
    \label{fig:t9t10-cd(b)}
\end{subfigure}
\hfill
\begin{subfigure}{0.32\textwidth}
{\scriptsize\[
\xymatrix@C=0.2cm@R=0.5cm{ 
\bx = \cdots 1110\gn{\diamond}\;\lllp  \,1\boldsymbol{1}1\gn{*}0\gn{\circ}\boldsymbol{1}   \gn{\prime}\fcolorbox{gray}{white}{01} 00\,\lllp\,\cdots\ar[d]    \\
\by = \cdots {?}010\;\lllp\,0{\bf{0}}\gn{\diamond}1\gn{*} 0\boldsymbol{0}0\gn{\circ}1\gn{\prime}00  \,\lllp\,\cdots  
}
\]}
    \caption{}
    \label{fig:t9t10-cd(c)} 
\end{subfigure}
\caption{The effect of $T_{9,10}$ on the switches in the case $y_1=0$ and
      (a) $D_{9,10}^r$, (b) $D_{9,10}^b$, (c) $D_{9,10}^{rb}$.}\label{fig:t9t10-cd}
\end{figure}

Let us first assume that $y_1=1$. In $D_{9,10}^r$ there are three r-switches at positions 3, 4 and 5. In the next time step, the first one moves two cells to the left and turns into a b-switch, whereas the others disappear (Figure~\ref{fig:t9t10(a)}). In $D_{9,10}^b$ there is only one b-switch at position 3, which moves two cells to the left (Figure~\ref{fig:t9t10(b)}).

In $D_{9,10}^{rb}$ there are two r-switches at positions 3 and 4 and a b-switch at position 5. In $\by$, the first moves two cells to the left and turns into a b-switch, whereas the other ones move two cells to the right and the b-switch turns into an r-switch (Figure~\ref{fig:t9t10(c)}).

Let us next assume that $y_1=0$. This must be the result of $T_9$, thus the considered domain $D_{9,10}$ overlaps the domain $D_{9,11}$ in such a way that $|x_{-3}x_{-2}x_{-1}x_{0}x_{1}x_{2}x_{3}|=D_{9,11}=|1110111|$ (Table~\ref{tab}).
But then, as shown in Figure~\ref{fig:t9t10-cd}, the number of switches in $V_{9,10}^r$ decreases, whereas it does not increase in $V_{9,10}^b$ and $V_{9,10}^{rb}$, since the switch at position~2 stems from the domain $D_{9,11}$ being overlapped. (It will be shown in the next subsection that the r-switch $^{\diamond}$ in $D_{9,11}=|1110\gn{\diamond}111|$ in such situations always moves two cells to the right.)

\begin{lem}\label{lem:t9t10}
The action of the pair $T_{9,10}$ does not produce new switches. Moreover, if $T_{9,10}$ acts on $D_{9,10}^r$, then it leads to a reduction in the number of switches.
\end{lem}

\subsection{The pair $T_{9,11}$} 

The consideration of the pair $T_{9,11}$ is a bit more involved, because there are three question marks in $V_{9,11}$ (see Table~\ref{tab-2}). Let $|x_1x_{2}x_{3}x_{4}x_{5}x_{6}x_{7}|= |1110111|= D_{9,11}$. The pair $T_{9,11}$ changes the values $x_2=x_5=1$ to $y_2=y_5=0$.

\vspace{-20 pt}
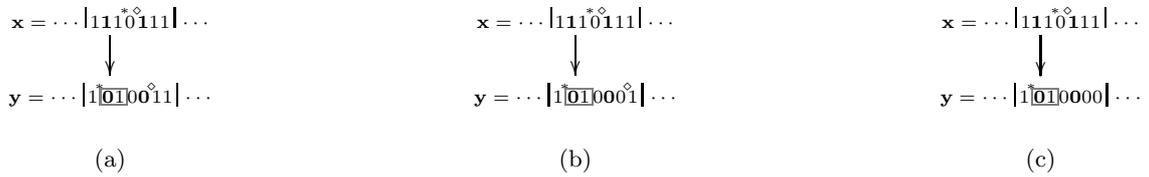
\begin{figure}[h!]
    \centering
    \begin{subfigure}{0.32\textwidth}
{\scriptsize\[
\xymatrix@R=0.5cm{ 
\bx = \cdots \;\lllp  1\boldsymbol{1}1\gn{*}0\gn{\diamond}\boldsymbol{1}11   \,\lllp\,\cdots\ar[d]  \\
\by = \cdots \;\lllp1\gn{*}\fcolorbox{gray}{white}{{\bf{0}}1}0\boldsymbol{0}\gn{\diamond}11  \,\lllp\,\cdots 
}
\]}
    \caption{}
    \label{fig:t9t11(a)}
\end{subfigure} 
\hfill
\begin{subfigure}{0.32\textwidth}
{\scriptsize\[
\xymatrix@R=0.5cm{ 
\bx = \cdots \;\lllp  1\boldsymbol{1}1\gn{*}0\gn{\diamond}\boldsymbol{1}11   \,\lllp\,\cdots\ar[d]  \\
\by = \cdots \;\lllp1\gn{*}\fcolorbox{gray}{white}{{\bf{0}}1}0\boldsymbol{0}0\gn{\diamond}1  \,\lllp\,\cdots 
}
\]}
\caption{}
\label{fig:t9t11(b)} 
\end{subfigure}
\hfill
\begin{subfigure}{0.32\textwidth}
{\scriptsize\[
\xymatrix@R=0.5cm{ 
\bx = \cdots \;\lllp  1\boldsymbol{1}1\gn{*}0\gn{\diamond}\boldsymbol{1}11   \,\lllp\,\cdots\ar[d]  \\
\by = \cdots \;\lllp{1}\gn{*}\fcolorbox{gray}{white}{{\bf{0}}1}0\boldsymbol{0}00  \,\lllp\,\cdots 
}
\]}
\caption{}
\label{fig:t9t11(c)} 
\end{subfigure}
\caption{The effect of $T_{9,11}$ on switches in $D_{9,11}$ in the case $y_1=1$ and
(a) $y_6y_7=11$, (b) $y_6y_7=01$, (c) $y_6y_7=00$.}
\end{figure}

\vspace{-20 pt}
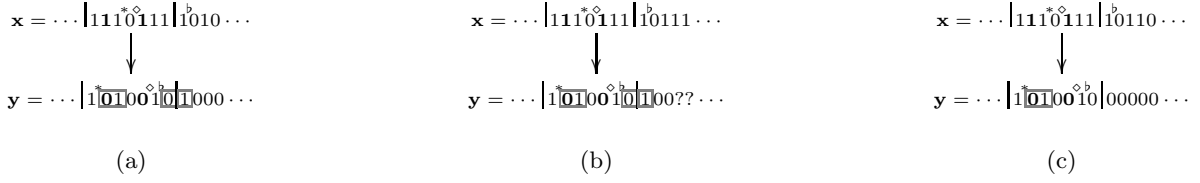
\begin{figure}[h!]
    \centering
    \begin{subfigure}{0.32\textwidth}
{\scriptsize\[
\xymatrix@R=0.5cm{ 
\bx = \cdots \;\lllp  1\boldsymbol{1}1\gn{*}0\gn{\diamond}\boldsymbol{1}11  \,\lllp\, 1\gn{\flat}010\cdots\ar[d]  \\
\by = \cdots \;\lllp1\gn{*}\fcolorbox{gray}{white}{{\bf{0}}1}0\boldsymbol{0}\gn{\diamond}1\gn{\flat}\fcolorbox{gray}{white}{01} \!\!\!\lllp \fcolorbox{gray}{white}{1} 000\cdots 
}
\]}
    \caption{}
    \label{fig:t9t11-cd(a)}
\end{subfigure} 
\hfill
\begin{subfigure}{0.32\textwidth}
{\scriptsize\[
\xymatrix@R=0.5cm{ 
\bx = \cdots \;\lllp  1\boldsymbol{1}1\gn{*}0\gn{\diamond}\boldsymbol{1}11  \,\lllp\, 1\gn{\flat}0111\cdots\ar[d]  \\
\by = \cdots \;\lllp1\gn{*}\fcolorbox{gray}{white}{{\bf{0}}1}0\boldsymbol{0}\gn{\diamond}1\gn{\flat}\fcolorbox{gray}{white}{01} \!\!\!\lllp \fcolorbox{gray}{white}{1} 00{?}{?}\cdots 
}
\]}
\caption{}
\label{fig:t9t11-cd(b)} 
\end{subfigure}
\hfill
\begin{subfigure}{0.32\textwidth}
{\scriptsize\[
\xymatrix@R=0.5cm{ 
\bx = \cdots \;\lllp  1{\boldsymbol{1}}1\gn{*}0\gn{\diamond}{\boldsymbol{1}}11 \,\lllp\,1\gn{\flat}0110\cdots\ar[d]  \\
\by = \cdots \;\lllp1\gn{*}\fcolorbox{gray}{white}{{\bf{0}}1}0{\boldsymbol{0}}\gn{\diamond}1\gn{\flat} 0 \,\lllp\, 00000\cdots 
}
\]}
\caption{}
\label{fig:t9t11-cd(c)} 
\end{subfigure}
\caption{The effect of $T_{9,11}$ on switches in $D_{9,11}$ in the case $y_1=1$ and $y_6y_7=10$, being overlapped by (a)
        $D_{9,10}$, (b) $D_{9,11}$, (c) $D_{9,12}$.}
\end{figure}

Let us first assume that $y_1=1$, then $y_2y_3=01$ is a box, so in $V_{9,11}$ there is a b-switch at position 1, which was previously an r-switch at position 3 in $D_{9,11}$. If $y_6y_7=11$, $y_6y_7=01$ or $y_6y_7=00$, then there may be at most one more switch in $V_{9,11}$, since the switch at position 4 in $D_{9,11}$ has simply moved to the right or disappeared (Figures~\ref{fig:t9t11(a)},~\ref{fig:t9t11(b)},~\ref{fig:t9t11(c)}, respectively).

If $y_6y_7=10$, then the change from $x_7=1$ to $y_7=0$ must be the result of $T_{9}$ (Table~\ref{tab}). Therefore, the considered domain is overlapped by one of the domains $D_{9,10}$, $D_{9,11}$ or $D_{9,12}$. In these situations, the switch at position 6 in $V_{9,11}$ stems from position 3 of the domain that overlap the considered domain (Figures~\ref{fig:t9t11-cd(a)},~\ref{fig:t9t11-cd(b)} and~~\ref{fig:t9t11-cd(c)}, respectively).

Next, let us assume that $y_1=0$, then the considered domain overlaps another $D_{9,11}$ in such a way that $|x_{-3}x_{-2}x_{-1}x_{0}x_{1}x_{2}x_{3}|=|1110111|$ (Table~\ref{tab}).
Also then the number of switches in the considered $V_{9,11}$ does not increase, since the switch at position 2 stems from the domain being overlapped, as we have shown during the analysis of $D_{9,10}$. 

\begin{lem}\label{lem:t9t11}
The action of the pair $T_{9,11}$ does not produce new switches. 
\end{lem}

\subsection{The pair $T_{9,12}$}

To describe the effect of the pair $T_{9,12}$ on the switches in $D_{9,12}=|x_1x_{2}x_{3}x_{4}x_{5}x_{6}x_{7}|$$=|1110110|$, we need to consider two variants of this domain:  $D_{9,12}^r=|1110110|$ with $x_{8}x_{9}x_{10}\neq 100$ and $D_{9,12}^b=|1110110100|$. 

\begin{figure}[h!b]
    \centering
    \begin{subfigure}{0.46\textwidth}
{\scriptsize\[
\xymatrix@C=0.2cm@R=0.5cm{ 
\bx = \cdots \;\lllp  1{\boldsymbol{11}}\gn{*}0\gn{\circ}{\boldsymbol{11}}\gn{\diamond} \ar[d]  0\,\lllp\,\cdots\ar[d]  \\
\by = \cdots \;\lllp1\gn{*}{\boldsymbol{00}}0{\boldsymbol{00}} 0\,\lllp\,\cdots 
}
\]}
    \caption{}
    \label{fig:t9t12(a)}
\end{subfigure}
\hfill
\begin{subfigure}{0.45\textwidth}
{\scriptsize\[
\xymatrix@C=0.2cm@R=0.5cm{ 
\bx = \cdots \;\lllp 1{\boldsymbol{11}}\gn{*}0\gn{\circ}{\boldsymbol{11}}\gn{\diamond} \ar[d]  0\,\lllp\,\cdots\ar[d]  \\
\by = \cdots \;\lllp0{\boldsymbol{00}}0{\boldsymbol{00}} 0\,\lllp\,\cdots 
}
\]}
    \caption{}
    \label{fig:t9t12(b)} 
\end{subfigure}
\caption{The effect of $T_{9,12}$ on the switches in $D_{9,12}^r$:
        (a) if $y_1=1$, (b) if $y_1=0$.}
\end{figure}
\begin{figure}[h!]
    \centering
    \begin{subfigure}{0.46\textwidth}
{\scriptsize\[
\xymatrix@C=0.2cm@R=0.5cm{ 
\bx = \cdots \;\lllp  111\gn{*}0\gn{\circ}{\boldsymbol{11}} \ar[d]  \gn{\diamond}\fcolorbox{gray}{white}{01}00\,\lllp\,\cdots\ar[d]  \\
\by = \cdots \;\lllp1\gn{*}000{\boldsymbol{00}}0  \gn{\circ}1\gn{\diamond}00\,\lllp\,\cdots 
}
\]}
    \caption{}
    \label{fig:t9t12-cd(a)}
\end{subfigure}
\hfill
\begin{subfigure}{0.45\textwidth}
{\scriptsize\[
\xymatrix@C=0.2cm@R=0.5cm{ 
\bx =\cdots \;\lllp  {111}\gn{*}0\gn{\circ}\boldsymbol{11} \ar[d]  \gn{\diamond}\fcolorbox{gray}{white}{01}00\,\lllp\,\cdots\ar[d]  \\
\by = \cdots \;\lllp0000\boldsymbol{00}0   \gn{\circ}1\gn{\diamond}00\,\lllp\,\cdots 
}
\]}
    \caption{}
    \label{fig:t9t12-cd(b)} 
\end{subfigure}
\caption{The effect of $T_{9,12}$ on the switches in $D_{9,12}^b$: 
        (a) if $y_1=1$, (b) if $y_1=0$.}
\end{figure}

In $D_{9,12}^r=|1110110|$ there are three r-switches at positions 3, 4 and 6. The latter two disappear and the first one moves two cells to the left if $y_1=1$, (Figure~\ref{fig:t9t12(a)}), while if $y_1=0$, then it also disappears (Figure ~\ref{fig:t9t12(b)}). The latter is due to the fact that if $x_1=1$ and $y_1=0$, then it must be the result of $T_{11}$, and so the considered domain overlaps with $D_{9,11}$ in such a way that $|x_{-3}x_{-2}x_{-1}x_{0}x_{1}x_{2}x_{3}|=|1110111|$ (Table~\ref{tab}) and we have explained all the switches that appear in $V_{9,11}$ (and none of them come from the right). 

In $D_{9,12}^b=|1110110100|$ there are two r-switches at positions 3 and 4 and one b-switch at position 6. The latter two move to the right and the b-switch turns into an r-switch. As for the first one, it moves to the left: if $y_1=1$, then two cells (Figure~\ref{fig:t9t12-cd(a)}), while if $y_1=0$, then more than two cells, thus landing in front of $V_{9,12}^r$ or disappearing (Figure ~\ref{fig:t9t12-cd(b)}).

\begin{lem}\label{lem:t9t12}
The action of the pair $T_{9,12}$ does not produce new switches. Moreover, the action of $T_{9,12}$ on $D_{9,12}^r$ leads to a reduction in the number of switches.
\end{lem}

\bigskip
The results of this section can be summarized as follows.

\begin{prop}\label{thm:c}
Let $\bx\in \xodd$. Then $s(F(\bx))\leq s(\bx)$, which means that the sequence $(s(\bx^t))^\infty_{t=0}$ is monotonically decreasing. Moreover, if $\bx$ contains $D_{5,6}^r$, $D_{7,8}^r$, $D_{9,10}^r$, $D_{9,12}^r$ or two blocks of 1s merge according to $T_{1,2}$ or $T_{3,4}$, then $s(F(\bx))<s(\bx)$.
\end{prop}

\section{The main theorem}
\label{sec:main_theorem}
In this section we prove that the BFO rule solves the parity problem by showing that it transforms any cyclic configuration of odd length into a homogeneous configuration after finitely many time steps (it suffices to prove this because the BFO rule  preserves parity). 

We start by proving some facts about configurations for which the number of switches remains constant during the entire evolution of the system. We denote the set of all such configurations by $\xconst$, \emph{i.e.} 
\[
\xconst = \{ \bx\in \xodd \mid {\rm the\;  sequence\;} \big(s(F^t(\bx))\big)^\infty_{t=0}\; {\rm is\; constant} \}.
\]
Therefore, if $\bx\in \xconst$, then 
\[
s(\bx)=s(F(\bx))=s(F^2(\bx))=s(F^3(\bx))=\ldots\,,
\]
thus also $F^t(\bx)\in \xconst$ for any $t\geq 0$. Of course, all homogeneous configurations from $\xodd$, like $0$, $1$, $000$, $111$, $00000$, $11111$ and so on, belong to $\xconst$. We will show that $\xconst$ does not contain any other configurations, and for this purpose we will rely upon the tool below.

\begin{defn}\label{def:OB}
    Let $\bx\in\xodd$. The block $x_{i}x_{i+1}\ldots x_{i+2k+1}$ of $\bx$, for some $k \in \mathbb{Z}_+$, will be referred to as \emph{an ordered block} if it satisfies the following three conditions:
\begin{itemize}
\itemindent=5pt
    \item[{\rm (C1)}] All two-symbol blocks $x_{i}x_{i+1}$, $x_{i+2}x_{i+3}$, ..., $x_{i+2k}x_{i+2k+1}$ belong to $\{00, 01,11\}$.
    \item[{\rm (C2)}] $x_{i}x_{i+1}=01$ and $x_{i+2k}x_{i+2k+1}\neq 01$.
    \item[{\rm (C3)}] If $x_{i+2k}x_{i+2k+1}=11$, then $x_{i+2k+2}=0$.
\end{itemize}
Furthermore, we say that an ordered block is \emph{maximal} if it is not a proper sub-block of another ordered block.
\end{defn}
In words, a block $x_{i}x_{i+1}\ldots x_{i+2k+1}$ is ordered if {\rm (C1)} it is formed by a concatenation of only the symbol pairs $00$, $01$ and $11$, {\rm (C2)} it begins with and does not end with $01$, and {\rm (C3)} if it ends with $11$, then it has to be followed by $0$. Note that condition {\rm (C2)} guarantees that $k\geq 1$, which means that an ordered block has a length of at least~4. 
Figure~\ref{fig:ordered} depicts sample configurations from $\xodd$ with some ordered blocks marked in gray.

\begin{figure}[h]
   \[
    \xymatrix@C=0.2cm@R=-0.03cm{ 
\colorbox{gray}{11101010100010100}10\colorbox{gray}{010000}110\colorbox{gray}{011}\\
\colorbox{gray}{111}0101010001010\colorbox{gray}{0100}10000110\colorbox{gray}{011}\\
0\colorbox{gray}{011111001101000011110111}0101010101010110  \\
\colorbox{gray}{0100}1110000000000000000001011\colorbox{gray}{0100000101011111}0\\
00011110011000000111111111111
}\]
    \caption{Sample configurations from $\xodd$ with selected ordered blocks marked in gray. Note that the last one does not contain any ordered block at all.}
    \label{fig:ordered}
\end{figure}

Two points are worth noting. First, ordered blocks occurring in a given configuration do not have to be disjoint, as the first two configurations in Figure~\ref{fig:ordered} illustrate, since they are actually the same; however, this fact can be ignored as it does not affect our considerations. Second, an ordered block cannot be too long compared to the length of the configuration, as stated below.
\begin{lem}\label{lem:OB}
An ordered block in $\bx\in X_{2n-1}$, $n \in \mathbb{N}$, may have a length of at most $2n$.
\end{lem}

\begin{proof}
Let $\bx\in X_{2n-1}$ for some $n \in \mathbb{N}$ and suppose, conversely, that $\bx = x_0x_1\ldots x_{2n-2}$ contains an ordered block $b$ having a length of at least $2n+2$. Without loss of generality, we may assume that $b$ started at position $0$, so
\[
b = \;\vline 01\vline x_2x_3\vline\;\ldots\;\vline x_{2n-2}0\vline 1x_2\vline \;\ldots\;\vline x_{2m-1}x_{2m}\vline \; ,
\]
with $m\geq 1$ (we have added vertical lines only to emphasize the boundaries between two-symbol blocks forming $b$).

Note that $x_{2n-2} = 0$, since no two-symbol block in $b$ can equal $10$ (see condition {\rm (C1)} in Definition~\ref{def:OB}). If so, we can define $k$ as the minimal number in the set $\big\{ i\in\{2,3,\ldots,2n-2\}\mid x_i=0\big\}$, as this set is not empty. This means that $x_1=\ldots=x_{k-1}=1$ and $x_k=0$, which entails that $k$ must be even (otherwise, again, some two-symbol block would be $10$). But then
\[
\vline 1x_2\vline \;\ldots\;\vline x_{k-1}x_{k}\vline = \;\vline 11\vline  11\vline \;\ldots\;\vline  11\vline 10\vline \; ,
\]
which shows that $2m$ cannot be greater than or equal to $k$ (as then $b$ would contain $10$), but also $2m$ cannot be smaller than $k$, as no ordered block can end with $11$ not followed by $0$ (see condition {\rm (C3)} in Definition~\ref{def:OB}). The obtained contradiction finishes the proof.
\end{proof}

Next, all the information about $\xconst$ obtained in the previous section is grouped together, with additional explanation of what it implies for ordered blocks.

\begin{prop}\label{thm:5-12}
    Let $\bx\in \xconst$. Then $\bx$ does not contain $D_{5,6}^r$, $D_{7,8}$, $D_{9,10}^r$ or $D_{9,12}^r$ and the action of $T_{1,2}$ and $T_{3,4}$ cannot lead to the merging of two blocks of 1s. Moreover, if $b=x_{i}x_{i+1}\ldots x_{i+2k+1}$ is an ordered block in $\bx$ and $\by=F(\bx)$, then there is an ordered block in $\by$ containing $y_{i}y_{i+1}\ldots y_{i+2k+1}$.
\end{prop}

\begin{proof}
The fact that $\bx$ does not contain $D_{5,6}^r$, $D_{7,8}^r$, $D_{9,10}^r$ or $D_{9,12}^r$ and that the action of $T_{1,2}$ and $T_{3,4}$ cannot lead to the merging of two blocks of 1s is obvious from Proposition~\ref{thm:c}. 
Moreover, if $\bx$ contains $D_{7,8}^b$ then, according to Lemma~\ref{lem:t7t8}, $F(\bx)$ contains $D_{5,6}^r$, thus implying that $s(F^2(\bx))<s(F(\bx))$, which is impossible since $F(\bx)\in \xconst$.

For proving the second statement, it suffices to look at Table~\ref{tab-3}, which presents a summary of the above considerations, namely, the only domains that can possibly occur in a configuration belonging to $\xconst$ along with their images.

\begin{table}[h]
	\centering
	\begin{tabular}{c}
{
\xymatrix@C=0.3cm@R=0.5cm{ 
D_{1,2} & D_{3,4} & D_{5,6}^b & D_{9,10}^b  \\
|\textcolor{gray}{ 111}\bf{\ensuremath\overset{{1}}{0}\ensuremath\overset{2}{0}}|\phantom{0} \ar[d] & |\textcolor{gray}{ 001}\bf{\overset{3}{0}\overset{4}{0}}|\phantom{0} \ar[d] & |\textcolor{gray}{ 0}{\bf\overset{5}{1}\overset{6}{1}}\textcolor{gray}{ 0100}| \ar[d] & |\textcolor{gray}{ 1}{\bf\overset{9}{1}}\textcolor{gray}{ 10}{\bf\overset{1\!0}{1}}\textcolor{gray}{ 00}|\ar[d]    \\
|\textcolor{gray}{?11}{\bf11}|\textcolor{gray}{ 0} & |\textcolor{gray}{??1}{\bf11}|\textcolor{gray}{ 0}          &   |\textcolor{gray}{?}{\bf{00}}\textcolor{gray}{0100}|          & |\textcolor{gray}{?}{\bf{0}}\textcolor{gray}{10}{\bf{0}}\textcolor{gray}{00}|   
}}
\\\\
{
\xymatrix@C=0.3cm@R=0.5cm{ 
 D_{9,10}^{rb} & D_{9,11} & D_{9,12}^b \\
           |\textcolor{gray}{1}{\bf\overset{{9}}{1}}\textcolor{gray}{{10}}{\bf\overset{1\!0}{1}}\textcolor{gray}{ 0100}|\ar[d] & |\textcolor{gray}{ 1}{\bf\overset{9}{1}}\textcolor{gray}{ 10}{\bf\overset{1\!1}{1}}\textcolor{gray}{ 11}|\ar[d]  & |\textcolor{gray}{ 1}{\bf\overset{9}{1}}{\bf\overset{1\!2}{1}}\textcolor{gray}{ 0}{\bf\overset{5}{1}\overset{6}{1}}\textcolor{gray}{0100}|\ar[d]   \\
 |\textcolor{gray}{?}{\bf{0}}\textcolor{gray}{10}{\bf{0}}\textcolor{gray}{0100}| & |\textcolor{gray}{?}{\bf{0}}\textcolor{gray}{10}{\bf{0}}\textcolor{gray}{??}| & |\textcolor{gray}?{\bf{00}}\textcolor{gray}0{\bf{00}}\textcolor{gray}{0100}| 
}}
\end{tabular}
\caption{The only domains that can possibly occur in a configuration belonging to $\xconst$ along with their images. The numbers in boldface shown on top of each domain are the AT numbers that act at those locations.}\label{tab-3}
\end{table}

Indeed, let $b=x_{i}x_{i+1}\ldots x_{i+2k+1}$ be an ordered block in $\bx\in \xconst$ and let $\by=F(\bx)$.
In order to prove that there is an ordered block in $\by$ containing $y_{i}y_{i+1}\ldots y_{i+2k+1}$, let us use Table~\ref{tab-3} to demonstrate that:
\begin{itemize}
    \item[(i)] Each two-symbol block $y_{i}y_{i+1}$, $y_{i+2}y_{i+3}$, ..., $y_{i+2k}y_{i+2k+1}$ belongs to $\{00, 01,11\}$.
    \item[(ii)] If the first two-symbol block $y_iy_{i+1}$ is not $01$, then it is preceded by $01$.
    \item[(iii)] If the last two-symbol block $y_{i+2k}y_{i+2k+1}$ is equal to $01$, then it is followed by $00$, and if it is equal to $11$, then it is followed by $0$ or $110$.
\end{itemize}

For proving (i), let us choose any two-symbol block $y_{i+2m}y_{i+2m+1}$, $m\in\{0,1,\ldots,k\}$, and suppose that $y_{i+2m}y_{i+2m+1}=10$. Since $b$ is an ordered block, it holds that $x_{i+2m}x_{i+2m+1}\in\{00, 01,11\}$. We consider each of these cases separately.
\begin{itemize}
\item[(a)] Suppose that $x_{i+2m}x_{i+2m+1}=00$ and $y_{i+2m}y_{i+2m+1}=10$. If $x_{i+2m}=0$ and $y_{i+2m}=1$, then it must be the result of the active transition $T_2$ or $T_4$; moreover, the two-symbol block $x_{i+2m}x_{i+2m+1}=00$ is preceded by $10$ (see Table~\ref{tab-3}). This is a contradiction, since neither $x_{i+2m}x_{i+2m+1}=00$ can be the beginning of $b$ (see condition {\rm (C2)}) nor the two-symbol block $x_{i+2m-2}x_{i+2m-1}=10$ can be a part of $b$ (see condition {\rm (C1)}).
   
\item[(b)]  Suppose that $x_{i+2m}x_{i+2m+1}=11$ and $y_{i+2m}y_{i+2m+1}=10$. If $x_{i+2m+1}=1$ and $y_{i+2m+1}=0$, then it must be the result of the active transition $T_9$; moreover, the two-symbol block $x_{i+2m}x_{i+2m+1}=11$ is followed by $10$ (see Table~\ref{tab-3}). This is a contradiction, since neither $x_{i+2m}x_{i+2m+1}=11$ can be the end of $b$, since it is followed by 1 (following condition {\rm (C3)}) nor the two-symbol block $x_{i+2m+2}x_{i+2m+3}=10$ can be a part of $b$ (see condition {\rm (C1)}).
    
\item[(c)]  Suppose that $x_{i+2m}x_{i+2m+1}=01$ and $y_{i+2m}y_{i+2m+1}=10$. This means that two active transitions are at play, the first one having to be the result of $T_2$ or $T_4$, but this means that $x_{i+2m-1}=0$. If so, then the second active transition must be $T_5$, which requires that the two-symbol block $x_{i+2m}x_{i+2m+1}=11$ is followed by $10$ (see Table~\ref{tab-3}). Again, this is a contradiction, since neither $x_{i+2m}x_{i+2m+1}=01$ can be the end of $b$ (see condition {\rm (C2)}) nor the two-symbol block $x_{i+2m+2}x_{i+2m+3}=10$ can be a part of $b$ (see condition {\rm (C1)}). 
\end{itemize}

In conclusion, the three contradictions above prove that for any $m$, $y_{i+2m}y_{i+2m+1}$ cannot be equal to $10$. 

Carrying on, in order to prove (ii), let us assume that $y_iy_{i+1} \neq 01$. Note that $y_iy_{i+1}$ cannot be $11$, as this would require the action of the pair $T_{1,2}$ or $T_{3,4}$ that would entail the merging of two blocks of ones, which contradicts the fact that $\bx\in \xconst$ (see Proposition~\ref{thm:c}). Hence, it necessarily holds $y_iy_{i+1} = 00$. However, since we know that $x_{i+2}x_{i+3} \neq 10$, if $x_ix_{i+1} = 01$ and $y_iy_{i+1} = 00$, this must be the result of $T_{10}$ or $T_{11}$, but in both cases $y_{i-2}y_{i-1}y_iy_{i+1} = 0100$.

Finally, to prove (iii), let us consider the end of the ordered block~$b$. Note that if $x_{i+2k}x_{i+2k+1}=00$, then $y_{i+2k}y_{i+2k+1}$ cannot be $01$ (since there is no active transition realizing this); also, if $y_{i+2k}y_{i+2k+1} = 11$, then $y_{i+2k+2}\neq 1$, because otherwise there would be a merging of two blocks of ones (and, as we know, this cannot happen, because $\bx\in \xconst$). Now, if $x_{i+2k}x_{i+2k+1}=11$ (and $x_{i+2k+2}=0$), then if $y_{i+2k}y_{i+2k+1} = 01$, this must be the result of $T_9$, but this implies that $y_{i+2k+2}y_{i+2k+3} = 00$, which means that $y_{i+2k}y_{i+2k+1}y_{i+2k+2}y_{i+2k+3} = 0100$; alternatively, if $y_{i+2k}y_{i+2k+1} = 11$, then $y_{i+2k+2}=1$ only if it is the result of $T_1$ or $T_3$, but then $y_{i+2k+2}y_{i+2k+3}y_{i+2k+4} = 110$, which means that $y_{i+2k}y_{i+2k+1}y_{i+2k+2}$ $y_{i+2k+3}y_{i+2k+4} = 11110$.

\end{proof}

In order to proceed, we need to establish some facts regarding ordered blocks.

\begin{lem}\label{lem:alter}
If $\bx\in \xconst$, then $\bx$ does not contain $010101$.
\end{lem}

\begin{proof}
Instead, let us suppose that $\bx$ does contain $010101$. Since the length of the configuration $\bx$ is odd, there exists $i\in\G$ such that $x_ix_{i+1}x_{i+2}x_{i+3}x_{i+4}x_{i+5}=010101$ and $x_{i-2}x_{i-1}\neq 01$. However, if $x_{i-1} = 0$, then $x_{i-1}x_ix_{i+1}x_{i+2}x_{i+3}x_{i+4}=001010 = D_{7,8}$, which is impossible according to Proposition~\ref{thm:5-12}. On the other hand, if $x_{i-2}x_{i-1} = 11 $, then, depending on $x_{i-3}$, $\bx$ contains $D_{5,6}^r$ or $D_{9,10}^r$, which is also impossible. 
\end{proof}

\begin{lem}\label{lem:D56}
Let $\bx\in \xconst$ and let $b$ be an ordered block in $\bx$. Then $b$ does not contain $001100$. Also, if $b$ is maximal, then it does not end with $0011$.
\end{lem}
\begin{proof}
For the first part, it is sufficient to note that if $b$ contains $001100$, then $\bx$ contains $D_{5,6}^r$. The same applies to the second part. Indeed, if $b=x_{i}x_{i+1}\ldots x_{i+2k+1}$ ends with $0011$, then, according to (iii), $x_{i+2k+2}=0$. Additionally, since $b$ is maximal, it holds that $x_{i+2k+3}=1$ and $x_{i+2k+4}x_{i+2k+5}\neq 00$, and hence $\bx$ contains $D_{5,6}^r$.
\end{proof}

\begin{lem}\label{lem:longer}
Let $\bx\in \xconst$ and let $b$ be any maximal ordered block in $\bx$. If $b$ ends with $11$, then $F(\bx)$ contains an ordered block longer than $b$.
\end{lem}
\begin{proof}
Without loss of generality, let $b=x_{0}x_{1}\ldots x_{2k+1}$ be a maximal ordered block in $\bx\in\xodd$ and $\by=F(\bx)$. If $x_{2k}x_{2k+1}=11$, then, according to Lemma~\ref{lem:D56}, $x_{2k-2}x_{2k-1}\neq 00$, and $x_{2k-1}=1$. Furthermore, condition (iii) guarantees that $x_{2k+2}=0$ and, since $b$ is maximal, $x_{2k+3}=1$. But this entails that $x_{2k-1}x_{2k}x_{2k+1}x_{2k+2}x_{2k+3}=11101$, which implies that in the next time step the AT pairs $T_{9,10}$, $T_{9,11}$ or $T_{9,12}$ will apply, leading to $y_{2k+2}y_{2k+3}=00$. Consequently, according to Proposition~\ref{thm:5-12}, $b$ (in $\bx$) will evolve into an ordered block $b'$ (in $\by$) containing at least $y_{0}y_{1}\ldots y_{2k+1}y_{2k+2}y_{2k+3}$.
\end{proof}

The following result is crucial. 

\begin{prop}
\label{lem:snakes}
If $\bx\in \xconst$, then $\bx$ does not contain any ordered block.
\end{prop}
\begin{proof}
First of all, let us recall that if $\bx\in\xconst$, then, according to Proposition~\ref{thm:5-12}, the ordered blocks in $\bx$ do not get shortened as time evolves.

Let us assume that some configurations from $\xconst$ contain an ordered block, from which we choose those that have the shortest length, and then, from them, we choose  those that contain the longest ordered block (whose length we denote by $l_0$). Let us denote the set of all such configurations by $B$. 
Let $\bx\in B$ and, without loss of generality, assume that $b=x_{0}x_1\ldots x_{2k+1}$ is the longest ordered block in $\bx$ (so, it is maximal). 

First we show that $b$ cannot end with $11$. 
Indeed, according to Lemma~\ref{lem:longer}, if $b$ ends with $11$, there should be a longer ordered block than $b$ in $F(\bx)$, contradicting our choice of $\bx$ and $b$.

Next, we analyse whether $b$ might end with $00$. Thus, suppose that $x_{2k}x_{2k+1}=00$ and let $x_{2m}x_{2m+1}$ be the last two-symbol block in $b$ that is not equal to $00$. If $x_{2m}x_{2m+1}=11$, then also  $x_{2m-1}=1$ (according to Lemma~\ref{lem:D56}), which means that the end of $b$ would look as follows: 
\[
b = \ ?\ldots?111\underbrace{00\ldots 00}_{2(k-m)}\ ;
\]
however, due to $T_{1,2}$, after $k-m$ time steps the ordered block $b$ would evolve into $b'$ in $F^{k-m}(\bx)$, having at least the same length as $b$ (so, $F^{k-m}(\bx)\in B$), but ending with a block of 1s, which, as already shown, would be a contradiction. Therefore, $b$ also cannot end with a block of at least two 0s which is preceded by at least two 1s.

So far, we have shown that if $B$ is non-empty, then every ordered block of maximum length contained in $\bx\in B$ ends with a block of 0s preceded by $01$. Now let us take such $\bx\in B$ in which an ordered block of maximum length $b=x_{0}x_1\ldots x_{2k+1}$ has the longest block of 0s at the end. This means that the end of $b$ would look as follows: 
\[
b=\ ?\ldots?01\underbrace{00\ldots 00}_{2(k-m)}\ .
\]
Note that $x_{2m}x_{2m+1}\ldots x_{2k}x_{2k+1}=0100\ldots 00$ cannot be preceded by $111$, as then there would be an ordered block $b'$ in $F(\bx)$, at least as long as $b$, and ending with the longer block of 0s:
\[
\phantom{ff}\bx=\ ?\ldots?11101\underbrace{00\ldots 00}_{2(k-m)}? \ , 
\]
\[
F(\bx)=\ ?\ldots?10100\underbrace{00\ldots 00}_{2(k-m)}? \ ,
\]
contradicting our choice of $\bx$ and $b$.

All that remains is to check all other possibilities for the part of $\bx$ preceding $x_{2m}x_{2m+1}\ldots x_{2k}x_{2k+1}=0100\ldots 00$, which leads to four possibilities:
\begin{itemize}\itemindent=15pt
    \item[Case 1:] If $x_{2m-1}=0$, due to $T_{3,4}$, after $k-m$ time steps the ordered block $b$ would evolve into $b'$ ending with $11$, which, as we already know, cannot happen.
    \item[Case 2:] If $x_{2m-3}x_{2m-2}x_{2m-1}=001$, then $\bx$ contains $D_{7,8}$, which is forbidden according to Proposition~\ref{thm:5-12}.
    \item[Case 3:] If $x_{2m-3}x_{2m-2}x_{2m-1}=011$, then due to $T_{5,6}$, $b$ would evolve into $b'$ in $F(\bx)$ with at least the same length as $b$, but ending with $00b$ and we are in Case 1.
    \item[Case 4:] If $x_{2m-3}x_{2m-2}x_{2m-1}=101$, then $x_{2m-4}=1$ (due to Lemma~\ref{lem:alter}) and $x_{2m-5}=1$ (otherwise, $\bx$ would contain $D_{5,6}^{r}$). But then $F(\bx)$ would contain $00b$, which leads again to Case 1.
\end{itemize}
\end{proof}

As a direct consequence, we have the following corollary.

\begin{cor}\label{cor:snakes}
If $\bx\in \xconst$, then $\bx$ does not contain $D_{5,6}^b$, $D_{9,10}^b$, $D_{9,10}^{rb}$, $D_{9,11}$ or $D_{9,12}^b$.
\end{cor}
\begin{proof}
It is sufficient to check that each of the domains $D_{5,6}^b$, $D_{9,10}^b$, $D_{9,10}^{rb}$ and $D_{9,12}^b$ contains an ordered block, and that the presence of $D_{9,11}$ in $\bx$ implies an ordered block in $F(\bx)$ (see Table~\ref{tab-3}).
\end{proof}

Using the above results, proving the main theorem is very simple.
\begin{thm}\label{main}
    The BFO rule solves the parity problem.
\end{thm}
\begin{proof}
As already mentioned, according to Theorem~\ref{thm:BFO1} it is sufficient to show that for each $\bx\in\xodd$ there exists $t\geq 0$ such that $F^t(\bx)$ is a homogeneous configuration. 

Let us then suppose, on the contrary, that there exists a configuration $\bx \in \xodd$ such that for each time step $t\geq 0$, the configuration $F^t(\bx)$ is not homogeneous, which, according to Proposition~\ref{prop:hom=0}, would mean that $s(F^t(\bx))\neq 0$. Since $\big(s(F^t(\bx))\big)^\infty_{t=0}$ is a monotonically decreasing sequence of natural numbers, there exists a time step $t_0$ such that
    \[
s(F^{t_0}(\bx))=s(F^{t_0+1}(\bx))=s(F^{t_0+2}(\bx))=s(F^{t_0+3}(\bx))=\dots =s\neq 0.    
    \]
Since no configuration in this sequence is homogeneous, at each time step at least one pair of active transitions has to apply (see Theorem~\ref{thm:BFO2}). 

However, according to Proposition~\ref{thm:5-12} and Corollary~\ref{cor:snakes}, none of the AT pairs $T_{5,6}$, $T_{7,8}$, $T_{9,10}$, $T_{9,11}$ and $T_{9,12}$ are permitted, so that at each time step it would be either $T_{1,2}$ or $T_{3,4}$ that should apply. But since both of these AT pairs increase the number of 1s in a configuration by two, this leads to a contradiction with the fact that the length of $F^{t_0}(\bx)$ is finite.
\end{proof}

\section{Concluding remarks}
\label{sec:conclusion}

The parity problem is an important benchmark problem in the setting of solving a global problem by means of a consensus that has to be achieved via totally local actions. As such, it has received continued attention in the automata network literature over the years, with various distinct formulations being proposed to solve it. Therefore, it was imperative to correct the original BFO rule, as it was the only proposed solution to the problem (so far) by means of a standard cellular automaton rule. 

Although the corrections we had to make to the active state transitions represented by $T_7$ and $T_8$ were trivial (since both patterns only had to be individually mirror reflected), this entailed the actual change in 24 specific state transitions of the original version.

As for the proof technique we employed, it departs from the original version in~\cite{BFO2013}, where a form of de Bruijn graph was used as the basic construct for the analyses. As explained, we essentially relied on the notion of symbol switches during the time evolution, especially in the context of the block $10100$ (referred to herein as the \emph{box}), the reason being that such a block is present (in reverse order) in both $T_7$ and $T_8$. Our approach required a large number of detailed analyses, but allowed us to track all possibilities of block creations and destructions until consensus is achieved. At first, we tried to preserve the original proof technique, but without success. As the notion of switches we ended up relying upon showed promise, we embarked on it. At the end, not only have we managed to solve the problem, but also realised that the abstract technique we crafted may also show its usefulness to tackle other related problems.

In addition to the correct proof given, the rule was successfully tested on all odd-sized initial configurations from 3 to 29. Now, with the fix we have provided, the BFO rule remains the only known single rule solution to the parity problem.

\section*{Acknowledgements}
\noindent
The work of Anna Nenca and Barbara Wolnik in this project is supported by Grant SONATA-18  of the National Science Centre, Poland, no.\ 2022/47/D/ST6/00416.
\noindent P.P.B. thanks Paola Flocchini and Heather Betel for theirr previous work on BFO; Pac\^{o}me Perrotin for computationally testing the updated version of the rule; and the Brazilian agencies CNPq for the research grant PQ 303356/2022-7, and CAPES for the research grants STIC-AmSud no.\ 88881.694458/2022-01 and Mackenzie-PrInt no.\ 88887.310281/2018-00.


\printbibliography
\end{document}